\documentclass{agtart_a}
\pdfoutput=1
\usepackage{amscd}

\title{Isovariant mappings of degree 1 and the Gap Hypothesis}

\author{Reinhard Schultz}
\givenname{Reinhard}
\surname{Schultz}
\address{Department of Mathematics\\
University of California at Riverside\\\newline
Riverside CA 92521\\
USA}
\email{schultz@math.ucr.edu}
\urladdr{}

\volumenumber{6}
\issuenumber{}
\publicationyear{2006}
\papernumber{28}
\startpage{739}
\endpage{762}

\doi{}
\MR{}
\Zbl{}


\received{29 September 2005}
\revised{8 May 2006}
\accepted{12 May 2006}
\published{12 June 2006}
\publishedonline{12 June 2006}
\proposed{}
\seconded{}
\corresponding{}
\editor{}
\version{}

\arxivreference{}   %%% not in arxiv

\let\xysavmatrix\xymatrix
\def\xymatrix{\disablesubscriptcorrection\xysavmatrix}
\AtBeginDocument{\let\bar\wbar\let\tilde\wtilde\let\hat\what}
\newcommand{\X}{\mskip1mu\relax}

\makeatletter
\def\cnewtheorem#1[#2]#3{\newtheorem{#1}{#3}[section]
\expandafter\let\csname c@#1\endcsname\c@thm}

\newtheorem{thm}{Theorem}[section]       
\cnewtheorem{lem}[thm]{Lemma}            
\cnewtheorem{prop}[thm]{Proposition}            
\newtheorem*{gaphyp}{Standard Gap Hypothesis}
\theoremstyle{definition}
\cnewtheorem{defn}[thm]{Definition}    
\cnewtheorem{exm}[thm]{Example}            
\cnewtheorem{exs}[thm]{Examples}            
\newtheorem*{rem}{Remark}            
            
\newtheorem*{ack}{Acknowledgments}            
\newtheorem*{claim}{Assertion}

\makeatother  %  move after \newtheorem block

\newcommand{\tinycirc}{{^{{\,}_{\text{\rm o}}}}}

\newcommand{\brh}{{}_{B\!R}H}

\begin{document}

\begin{asciiabstract}
Unpublished results of S Straus and W Browder state that two notions
of homotopy equivalence for manifolds with smooth group actions -
isovariant and equivariant - often coincide under a condition called
the Gap Hypothesis; the proofs use deep results in geometric topology.
This paper analyzes the difference between the two types of maps from
a homotopy theoretic viewpoint more generally for degree one maps if
the manifolds satisfy the Gap Hypothesis, and it gives a more homotopy
theoretic proof of the Straus-Browder result.
\end{asciiabstract}

\begin{abstract}
Unpublished results of S Straus and W Browder state that two notions
of homotopy equivalence for manifolds with smooth group
actions---isovariant and equivariant---often coincide under a
condition called the Gap Hypothesis; the proofs use deep results in
geometric topology.  This paper analyzes the difference between the
two types of maps from a homotopy theoretic viewpoint more generally
for degree one maps if the manifolds satisfy the Gap Hypothesis, and
it gives a more homotopy theoretic proof of the Straus--Browder
result.
\end{abstract}

\maketitle

In order to motivate and explain our results more thoroughly, we
shall begin with various pieces of background material.  Readers
who would rather focus on the main results and their proofs may
go directly to \fullref{stMnRes}.

\section{Background}

Ever since the topological classification of surfaces was discovered, 
one basic theme in geometric topology has been the reduction of existence 
and classification questions for manifolds to problems in algebraic
topology.  A collection of techniques known as \emph{surgery theory}
has been particularly effective in this regard (compare Rothenberg \cite[pages
375--378]{R}).  For well over four decades topologists have also known 
that such techniques also have far reaching implications for manifolds 
with group actions (cf Browder \cite{Br1} and \cite[pages 378--379]{R}). 
Not surprisingly, many of the most striking applications of surgery 
theory require some assumption on the manifolds, mappings or 
structures under consideration, and for group actions the following
restriction has been employed quite often:

\begin{gaphyp}
For each pair of isotropy subgroups $H \supsetneqq K$ and each 
pair of components $B\subset M^H$, $C\subset M^K$ such that
$B\subsetneqq C$ we have $ \dim B+1 \leq {\frac{1}{2}}(\dim C )$.
\end{gaphyp}

A condition of this sort first appeared explicitly in the 
unpublished Berkeley doctoral dissertation of Sandor H 
Straus \cite{St}, and the importance and usefulness of the 
restriction became apparent in work of T Petrie
\cite{P1,P2}; see also Dovermann--Petrie \cite{DP}, 
Dovermann--Rothenberg \cite{DR}, and L\"uck--Madsen \cite{LuMa}.  
The reason for this is that the Gap Hypothesis automatically 
implies that certain maps are \emph{isovariant} (cf Palais \cite{Pa}):  
A mapping of $f\co X\to Y$ of spaces with actions
of a group $G$ is said to be isovariant if it is equivariant --- so 
that $f(g\cdot x) = g\cdot f(x)$ for all $g\in G$ and $x\in X$ --- and
for each $x$ the isotropy subgroup $G_x$ of all group elements fixing 
$x$ is equal to the isotropy subgroup $G_{f(x)}$ of the image point
(in general one can only say that the first subgroup is contained in
the second).  Specifically, surgery theory yields equivariant mappings 
of the form $G/K\times S^q\to G\cdot M^K$ where $q\leq {\frac{1}{2}}
\dim M^K$, and at one crucial step in the process one uses the fact
that such maps are equivariantly homotopic to isovariant maps by general 
position if the Gap Hypothesis holds. Usually this is expressed
in other terms (eg, the image of $\{H\}\times S^q$ lies in 
$G\cdot M^H$ and is
disjoint from $G\cdot M^L$ for all isotropy subgroups properly
containing a conjugate of $H$), but it is straightforward to verify 
that the usual version is equivalent to the one stated here in
terms of isovariance.

Applications of surgery to group actions that do not require the Gap
Hypothesis frequently assume that the underlying maps of manifolds 
are isovariant or \emph{almost isovariant} (cf Browder--Quinn
\cite{BQ}, the final section of Dula--Schultz \cite{DuS}, the second 
part of Schultz \cite{Sc3}, and Weinberger's book \cite{We}).  The general 
notion of almost isovariance is defined precisely on page 27 of
\cite{DuS}, and the most important special case is reproduced below.  
For the time being, we merely note that

\begin{enumerate}
\item
the sets of nonisovariant points ($G_x\neq G_{f(x)}$) for such a map
may be pushed into very small pieces of the domain where they cause
no problems,

\item
standard methods of homotopy theory extend directly to a suitably 
defined category of almost isovariant mappings (cf \cite{DuS} 
and Dror Farjoun \cite{DF}),

\item
results of \cite{DuS} show that almost isovariant homotopy and isovariant
homotopy are equivalent in many important cases (including all 
smooth actions of finite cyclic $p$--groups), and a standard conjecture
(believed by most workers in the area) states that the same is true 
for arbitrary smooth actions of finite groups.
\end{enumerate}

\noindent
In fact, isovariant techniques play a central role in classification
results for group actions when the Gap Hypothesis fails; in many cases
where such machinery is not used explicitly, the work can readily be 
interpreted in these terms.  A fundamentally important breakthrough in 
this area was due to Browder and Quinn \cite{BQ} (see also the commentary 
on the latter in Hughes--Weinberger \cite{HW}), and a more general discussion 
of the situation in the smooth category --- which also extends earlier work on 
the smooth classification of topologically linear actions on 
spheres --- appears in \cite[Section~II.1]{Sc3} (see also the final 
section of \cite{DuS}).  In the piecewise 
linear and topological categories, there is a distinct body of results
which is largely based on techniques from controlled topology
(eg Weinberger's book \cite{We}, the survey articles by 
Hughes--Weinberger \cite{HW} and Cappell--Weinberger \cite{CW}, and 
the doctoral dissertation of A Beshears \cite{Bsh}).  A full 
historical account of the topic is beyond the scope of this paper, 
but some additional references include the work of A Assadi with 
Browder \cite{AsB} and P Vogel \cite{AV}, the monograph by L Jones 
on symmetries of disks \cite{J}, and the material on symmetries of 
aspherical manifolds in Weinberger's paper on higher Atiyah--Singer 
invariants \cite{We2}.

\subsection*{Focus of this paper}

In this paper we are particularly interested in the following 
unpublished result, which is due to Straus \cite{St} for actions 
that are semifree (the group acts freely off the fixed point set)
and W Browder \cite{Br2} more generally. It implies a fairly strong,
general and precise connection between almost isovariance and the 
Gap Hypothesis.

\begin{thm}\label{thm1}
Let $f\co M\to N$ be an equivariant homotopy equivalence of 
connected, compact, unbounded \textup{( =~~{closed})}
and oriented smooth $G$--manifolds that 
satisfy the Gap Hypothesis.  Then $f$ is equivariantly homotopic to 
an almost isovariant homotopy equivalence.  
\end{thm}

As noted above, in some cases the results of \cite{DuS} allow one to
replace ``almost isovariant'' by ``isovariant'' in the conclusion; 
in particular, this is true if the isotropy subgroups are normal and
linearly ordered by inclusion.

Although \fullref{thm1} is a purely homotopy theoretic statement,
the proofs in \cite{St} and \cite{Br2} require fairly deep results 
from Wall's nonsimply connected surgery theory \cite{Wl}, which in 
turn depends upon other deep geometric results such as the 
classification theory for immersions (cf Phillips \cite{Ph}) 
and the Whitney process for geometrically eliminating pairs of double 
points that cancel algebraically (cf 
Milnor's monograph on the $h$--cobordism theorem \cite{Mi}).  It 
is natural to ask
whether one can prove \fullref{thm1} without relying so extensively on 
such a large amount of auxiliary material (this is a special case of 
the classical scientific maxim called \emph{Ockham's Razor}).  In 
particular, since one can construct a version of obstruction theory 
for isovariant
maps and define obstructions to finding an isovariant deformation of
a given equivariant map \cite{DuS}, it is natural to search for a proof 
that is related to this obstruction theory.  More generally, one
would also like to understand the obstructions to isovariance for 
arbitrary equivariant mappings of degree $\pm 1$ from one smooth 
manifold to another.  Some basic test cases are examples of equivariant 
degree one mappings mentioned in \cite{Br2} which are not equivariantly 
homotopic to isovariant maps; results of K\,H Dovermann on isovariant 
normal maps \cite{Dov} also provide some motivation.

The main objective of this paper is to analyze the problem of 
deforming an equivariant degree one map into an isovariant map 
when the Gap Hypothesis holds, to use this criterion to provide 
an essentially homotopy-theoretic proof of \fullref{thm1}, and to see 
how the criterion applies to equivariant homotopy equivalences
and other basic examples.  In contrast to \cite{St} and \cite{Br2}, 
our approach requires a minimum of input from geometric topology; 
namely, \emph{nonequivariant} transversality and standard results
on smooth embeddings in the general position range.  For the sake
of clarity we shall restrict attention to finite group actions
that are semifree in the sense described above; if $G$ is cyclic
of prime order, then all actions are semifree.  We shall also 
discuss some applications of \fullref{thm1} to cyclic reduced products
that were first considered in \cite{St} and a few positive and 
negative results just outside the range of the Gap Hypothesis 
(further information on the latter will appear in sequels to 
this paper).

\begin{ack} 
I am extremely grateful to Bill Browder 
for helpful conversations and correspondence regarding his results 
on the questions treated here, and especially for providing a detailed 
account of his counterexamples that appear in \fullref{normStraight} of this paper.   
I would also like to thank Heiner Dovermann for various conversations 
involving his work.  Comments on the referee for enhancing the exposition 
of the paper were also valuable and appreciated.  The research in this 
paper was partially supported by National Science Foundation Grants DMS 
86-02543, 89-02622 and 91-02711, and also by the Max--Planck--Institut 
f\"ur Mathematik in Bonn, and sources are also gratefully acknowledged.
\end{ack}

\section{Statements of main results}\label{stMnRes}

Suppose that $M$ and $N$ are closed, 
oriented, semifree smooth $G$--manifolds satisfying the Gap Hypothesis 
such that all components of the fixed point sets, and suppose that 
$f\co M\to N$ is a $G$--equivariant map of degree 1.  Let $\{N_\alpha\}$ 
denote the set of components of the fixed point set $N^G$, where we 
may as well assume 
that $\alpha$ runs through the elements of $\pi_0(N^G)$, and suppose 
that the associated map $f^G$ of fixed point sets defines a 1--1 
correspondence between the components of $M^G$ and $N^G$; for 
each $\alpha$ let 
$$M_\alpha~~=~~f^{-1}(N_\alpha)~\cap~M^G$$
and let $f_\alpha\co M_\alpha\to N_\alpha$ denote the partial 
map of fixed point sets determined by $f$.
Denote the equivariant normal bundles of $M_\alpha$ and 
$N_\alpha$ in $M$ and $N$ by $\xi_\alpha$ and $\omega_\alpha$
respectively, and let $S(\nu)$ generically represent the unit
sphere bundle of the vector bundle $\nu$ (with the associated
group action if $\nu$ is a $G$--vector bundle).

\begin{thm}\label{thm2}
Suppose we are given the setting above such that $\dim M_\alpha = 
\dim N_\alpha$ for each $\alpha$.  

{\rm(i)}\qua If $f$ is homotopic to an isovariant map, then for
each $\alpha$ the map $f_\alpha$ has degree $\pm 1$, and 
$S(\xi_\alpha)$ is equivariantly fiber homotopy equivalent 
to $S(f^*\omega_\alpha)$. 

{\rm(ii)}\qua If the two conditions in the preceding statement hold,
then $f$ is equivariantly homotopic to a map that is isovariant
on a neighborhood of the fixed point set.

{\rm(iii)}\qua If $f$ is isovariant on a neighborhood of the fixed point
set, then $f$ is equivariantly homotopic to an isovariant map if
and only if $f$ is equivariantly homotopic to a map $f_1$
for which the set of nonisovariant points 
of $f_1$ is contained in a tubular neighborhood of $M^G$.
\end{thm}

\fullref{thm1}  will follow immediately from \fullref{thm2} and the
specialization of the the latter to equivariant homotopy
equivalences.

\begin{thm}\label{thm3}
In the setting of the previous result, suppose that $f$ is an
equivariant homotopy equivalence.  Then $f$ is equivariantly 
homotopic to an isovariant homotopy equivalence.
\end{thm}

\textbf{Note}\qua If $f$ is an equivariant homotopy equivalence,
then some of the assumptions in the setting and statement of
\fullref{thm2} are redundant because an equivariant homotopy
equivalence determines a 1--1 correspondence of fixed point
set components and the dimensions of corresponding components 
are also equal because the components are homotopy equivalent.

\subsection*{Application to deleted reduced products}

Given a topological space $X$ and an integer $n$, the $n$--fold 
cyclic reduced product is defined to be the quotient
of the product space $X^n$ (ie $n$ copies of $X$)
modulo the action of ${\mathbb Z}_n$ on the latter by permuting
coordinates, and the deleted cyclic reduced product is the subset
of the latter obtained by removing the image of the diagonal
$\Delta(X^{n})$ consisting of those points whose coordinates
are all equal.  In his thesis \cite{St} Straus used his version of 
\fullref{thm1} to obtain the following homotopy invariance 
property for such spaces:

\begin{thm}\label{thm4}
Let $M$ and 
$N$ be closed smooth manifolds of dimension $\geq 2$, let $p$ be 
an odd prime, and suppose that $M$ and $N$ are homotopy equivalent.  
Let ${\mathbb Z}_p$ act smoothly on the $p$--fold self products $\Pi^p M$ 
and $\Pi^pN$ (where $\Pi^pX= X\times\cdots\times X,\ p\ factors$) by 
cyclically permuting the coordinates, and let ${\text{\bf D}}^p(M)$, 
${\text{\bf D}}^p(N)$ be the invariant subsets sets given by removing 
the diagonals from $\Pi^pM$ and $\Pi^pN$. Then the deleted reduced 
cyclic products ${\text{\bf D}}^p(M)/{\mathbb Z}_p$ and 
${\text{\bf D}}^p(N)/{\mathbb Z}_p$ are homotopy equivalent.
\end{thm}

As also noted in \cite{St}, this result does not extend to compact 
bounded manifolds, and in fact closed unit disks of different 
dimensions generate simple counterexamples.  The results of Schultz
\cite{Sc2} imply that \fullref{thm4} extends to simply connected manifolds  
if $p=2$, but recent results of R Longoni and P Salvatore \cite{LS} 
imply that the result does not extend to 3--dimensional lens spaces
when $p = 2$.  Further results on the relationship between
${\text{\bf D}}^2(M)/{\mathbb Z}_2$ and ${\text{\bf D}}^2(N)/
{\mathbb Z}_2$ for homotopy equivalent manifolds appear in a
paper by P L\"offler and R\,J Milgram \cite{LoMi}.

Since the statement and proof of \fullref{thm4} have apparently not 
appeared previously in print, we shall outline the (fairly 
straightforward) argument for the sake of completeness.

\begin{proof}[Sketch of proof]  
We shall first prove the result when $\dim M = \dim N \geq 3$ and
then dispose of the remaining cases afterwards.
If $f\co M\to N$ is a homotopy 
equivalence then $\Pi^pf\co \Pi^pM\to\Pi^pN$ is an equivariant 
homotopy equivalence of closed smooth ${\mathbb Z}_p$--manifolds.  
All actions of ${\mathbb Z}_p$ are semifree if $p$ is prime, so this 
condition holds automatically.  Furthermore, the fixed point sets
of the action on $\Pi^pM$ and $\Pi^pN$ are the respective diagonals 
$\Delta(\Pi^pM)\cong M$ and $\Delta(\Pi^pN)\cong N$, and since
$$\dim\Delta^p(\Pi^pX)~~=~~\dim X~~=~~\left(\dim \Pi^pX\right)/p~~\leq$$
$${\textstyle{\frac{1}{3}}} \dim \Pi^pX~~<~~
{\textstyle{\frac{1}{2}}}\dim \Pi^pX~-~1$$
if $X = M$ or $N$ is at least 3--dimensional and $p$ is odd, then 
the Gap Hypothesis also holds.
Therefore \fullref{thm1} implies that $\Pi^pf$ is
equivariantly homotopic to an isovariant homotopy equivalence,
and the latter in turn yields an equivariant homotopy equivalence 
from ${\text{\bf D}}^p(M)$ to ${\text{\bf D}}^p(N)$.  The induced 
map of orbit spaces is the desired homotopy equivalence from
${\text{\bf D}}^p(M)/{\mathbb Z}_p$ to ${\text{\bf D}}^p(N)/
{\mathbb Z}_p$.

Suppose now that $\dim M = \dim N \leq 2$.  In these cases
homotopy equivalent manifolds are homeomorphic, so we can take
the homotopy equivalence $f\co M\to N$ to be a homeomorphism.  It 
follows immediately that $\Pi^p f$ is a homeomorphism and as 
such is automatically isovariant.  One can now complete the proof 
as in the last two sentences of the preceding paragraph.
\end{proof}

\subsection*{Further results}

The results of \cite{Br2} also include a uniqueness statement (up
to isovariant homotopy) if $M \times [0,1]$ and $N\times [0,1]$ 
satisfy the Gap Hypothesis.  One can also use the methods of this
paper together with some additional geometric and homotopy theoretic 
input to prove such a uniqueness result.  The necessary machinery to 
do so will be developed in a subsequent paper.

Since the results of \cite{Br2} also apply to actions that are not
necessarily semifree, it is natural to ask whether the methods
of this paper extend.  The answer is yes, but a proof would
require the introduction of a considerable amount of extra
machinery that would take time to develop and might obscure
the main ideas, and this is a major reason for sticking with
the semifree case.  A brief discussion of some tools needed
to carry out such extensions appears at the end of \fullref{Prelim},  
and we plan to pursue this further in another sequel to this
article.

\subsection*{Overview of the paper}

We shall begin \fullref{Prelim} by proving that the conditions in \fullref{thm2}
are necessary for a map $f$ as above to be properly homotopic to
an isovariant map.  The proof of sufficiency in \fullref{thm2} splits into 
two steps, which are carried out in Sections \ref{Prelim} and \ref{normStraight}.  To motivate 
the first step, observe that an equivariant map of 
semifree $G$--manifolds is automatically isovariant on the fixed point 
set, so a natural starting point is to determine whether the given
map can be equivariantly deformed to a map that is isovariant on a
neighborhood of the fixed point set.  If this is possible and we
have a map with this additional property, the next step is to determine
whether such a map can be further deformed to another one which is
isovariant everywhere.  \fullref{equivar} contains the 
proofs of Theorems \ref{thm1} and \ref{thm3}.  Finally, in \fullref{FinRem}
 we shall discuss a variety of questions related to the main results.  
Some concern the interaction between equivariance and isovariance
when the Gap Hypothesis holds, while others involve borderline 
situations in which the Gap Hypothesis inequalities fail, but only
by a small amount.  There are obvious questions about the extent
to which \fullref{thm1} extends to such examples, and answers 
to such questions turn out to have close connections to issues in 
equivariant surgery theory in such borderline cases.  Sometimes
the latter yields generalizations of \fullref{thm1} to situations
where the conclusion is not particularly obvious from a 
homotopy theoretic viewpoint, and in other cases homotopy theory
has implications for equivariant surgery theory.

\section{Preliminary adjustments and necessity}\label{Prelim}

It will be convenient to begin with some notational conventions and
elementary observations in order to simplify the main discussion and 
the proofs.

Let $P$ be a closed smooth $G$--manifold, where $G$ is a 
finite group.  By local linearity of the action we know that the fixed
point set $P^G$ is a union of connected smooth submanifolds;
as before, denote these connected components by $P_\alpha$.  
For each $\alpha$ let $D(P_\alpha)$ denote a closed tubular neighborhood.
By construction these sets are total spaces of closed unit disk bundles
over the manifolds $P_\alpha$, so let $S(P_\alpha)$ and
denote the associated unit sphere bundles; it follows that
$$\partial [D(P_\alpha)]~~=~~S(P_\alpha)~.$$
Suppose now that $M$ and $N$ are smooth semifree $G$--manifolds 
and $f\co M\to N$ is an equivariant mapping.

\subsection*{The ``only if'' direction of \fullref{thm2}}

Assume that we have the setting and notation introduced in order 
to state \fullref{thm2}:

\begin{enumerate}
\item
$M$ and $N$ are compact, oriented, semifree 
smooth $G$--manifolds satisfying the Gap Hypothesis.

\item
$f\co M\to N$ is a $G$--equivariant map of degree 1. 

\item
$\{N_\alpha\}$ denotes the set of components of 
$N^G$ where we may as well assume that $\alpha$ runs through 
the elements of $\pi_0(N^G)$.

\item
The associated map $f^G$
of fixed point sets defines a 1--1 correspondence between the 
components of $M^G$ and $N^G$. 

\item
If for each $\alpha$ we let 
$$M_\alpha~~=~~f^{-1}(N_\alpha)\cap M^G$$
then $f_\alpha$ is the continuous map from $M_\alpha$ to $N_\alpha$ 
determined by $f$.

\item
If the equivariant normal bundles of $M_\alpha$ and 
$N_\alpha$ in $M$ and $N$ are $\xi_\alpha$ and $\omega_\alpha$
respectively, and let $S(\nu)$ and $D(\nu)$ generically represent 
the unit sphere and disk bundle of the vector bundle $\nu$ (with 
the associated group action since $\nu$ is a $G$--vector bundle).
\end{enumerate}

Not surprisingly, we shall also use the notational conventions
we have previously developed and mentioned.

\begin{proof}[Necessity proof for \fullref{thm2}]
Each of the first two basic conditions  
in \fullref{thm2} depends only
on the equivariant homotopy class of a mapping of manifolds, 
so without loss of generality we may replace $f$ by any map 
in the same proper equivariant homotopy class.  In particular, 
since we are assuming that $f$ is properly equivariantly homotopy 
to an isovariant map, we might as well assume that $f$ itself is 
isovariant.

By the results of \cite{DuS} (in particular, see Proposition 4.1 
on page 27), the map $f$ is isovariantly homotopic to a map $f_0$ 
such that for each $\alpha$ we
have $f(\X D(M_\alpha)\,)\subset D(N_\alpha)$,
$f(\X S(M_\alpha\,))\subset S(N_\alpha)$, and
$$f\left(\X M - \cup_\alpha~{\text{\rm Int\,}} D(M_\alpha)\X
\right)~~\subset~~
N - \cup_\alpha~{\text{\rm Int\,}} D(N_\alpha)~.$$
For each choice of $\alpha$ let 
$$h_\alpha: \bigl(\, D(M_\alpha),S(M_\alpha)\,\bigr)~\longrightarrow~
\bigl(\, D(M_\alpha),S(M_\alpha)\,\bigr)$$ 
be the associated map of pairs 
defined by $f_0$.  Since the latter has degree 1, the same 
is true for each of the maps $h_\alpha$.  We have already noted 
that $D(-)$ and $S(-)$ are disk and sphere bundles over the 
appropriate components of fixed point sets, and therefore a simple 
spectral sequence argument implies that $(a)$ the degrees of the 
maps $f_\alpha$ are all equal to $\pm 1$, up to an equivariant
homotopy of pairs the map $h_\alpha$ sends a spherical fiber in 
$S(M_\alpha)$ to a spherical fiber in $S(N_\alpha)$ by a map of 
degree $\pm 1$.  Therefore an equivariant analog of a classical 
result due to A Dold \cite{Dol} (cf Waner \cite{Wn})
shows that there is a $g$--equivariant fiber homotopy equivalence 
from $S(\xi_\alpha)$ to $S(f_\alpha^*\,\omega_\alpha)$, where as 
before $\xi_\alpha$ and $\omega_\alpha$ denoted the corresponding 
equivariant normal bundles for $M_\alpha$ and $N_\alpha$.  This 
completes the proof that Condition (i) holds. Since the set of 
nonisovariant points for an isovariant map is empty by definition, 
Condition (iii) is automatically true, so the proof is complete.
\end{proof}

\subsection*{Some examples}

It is not difficult to construct equivariant maps of degree 
1 which satisfy the Gap Hypothesis but do not satisfy the 
statements in Condition (i) of \fullref{thm2} on degrees 
and equivariant normal bundles.

\begin{exm}\label{ex1}
Let $V$ be a nontrivial semifree real representation of $G$ 
such that $\dim V^G > 0$ and the Gap Hypothesis holds, and 
let $S^V$ be the one point compactification, which is equivariantly 
homeomorphic to the unit sphere in $V\oplus {\mathbb R}$.  It is well 
known that for each positive integer $k$ there is a $G$--equivariant 
map $h_k\co S^V\to S^V$ such that $\deg h_k = 1$ and $\deg h_k^G = 
k|G|+1$ (eg, this is a very special case of the equivariant 
Hopf Theorem stated in tom Dieck's book on the Burnside ring and 
equivariant homotopy theory \cite[Theorem~8.4.1, pages~213--214]{tD}).  
Since the fixed point set is connected and the 
degree of the map on the fixed point set is not $\pm 1$ if $|G|>2$ 
or $k \neq 0$, it follows that $h_k$ cannot be homotopic to an 
isovariant map.  However, the map $h_k$ does satisfy the second 
part of Condition (i) involving pullbacks of equivariant normal 
bundles because the equivariant normal bundle of $(S^V)^G$ in $S^V$ 
is a product bundle.
\end{exm}

\begin{exm}\label{ex2}
Let $G$ be a cyclic group, assume that $k$, $m$ and $r$ satisfy 
$k\equiv 0 (4)$, $k, m > 0$  and $2r > m + k$, and let $\gamma$ 
be a complex $r$--dimensional vector bundle over $S^k$ which 
represents a generator of $\pi_k(BU_r)\cong {\mathbb Z}$.  Take $M$ 
to be the associated $(2r + m)$--sphere bundle over $S^k$.   Then 
$G$ acts smoothly and fiber preservingly on $M$ with fixed point 
set $S^k\times S^m$, and each point has an invariant open 
neighborhood on which the action of $G$ is smoothly equivalent 
to the linear representation $V = {\mathbb R}^{k+m}\oplus {\mathbb C}^r$.  
If we collapse everything outside such a neighborhood to a point, 
we obtain an equivariant map of degree 1 from $M$ to $S^V$.  This 
map also has degree 1 on the fixed point sets, but we claim it 
does not satisfy the pullback condition for equivariant normal 
bundles.  Since the equivariant normal bundle in $M$ is the pullback 
of $\gamma$ under the coordinate projection map from $S^k\times 
S^m\to S^m$, it will suffice to check that $\gamma$ is not 
equivariantly fiber homotopically trivial.  In fact, the underlying 
\emph{nonequivariant} vector bundle is well known to be stably fiber 
homotopically \emph{nontrivial} (see, for example, Adams \cite{Ad}).
\end{exm}

\section{Normal straightening and relative isovariance obstructions}\label{normStraight}

In this section we shall prove the implication of \fullref{thm2} 
in the other direction; namely, if Conditions (ii) is satisfied then 
one can make $f$ isovariant near the fixed point set, and if (iii) is 
satisfied then the map $f$ is equivariantly homotopic to an isovariant 
mapping.  The first step is to examine the consequences of Condition (ii).

\begin{prop}\label{prop41}
In the setting of \fullref{stMnRes}, suppose that $f\co M\to N$
is a continuous equivariant degree $1$ map.
Assume that each of the associated maps of fixed 
point components $f_\alpha\co M_\alpha\to N_\alpha$ has degree $\pm 1$ 
and that $S(\xi_\alpha)$ is equivariantly fiber homotopy equivalent 
to $S(f^*\,\omega_\alpha)$ for each $\alpha$.  Then there are closed,
pairwise disjoint, equivariant tubular neighborhoods $D(M_\alpha)$
of the fixed point set components $M_\alpha$ and an equivariant
mapping $f_0$ such that $f_0$ is equivariantly homotopic to $f$ and 
for each $\alpha$ the restriction $f|D(M_\alpha)$ is isovariant.
\end{prop}

\begin{proof}
For each $\alpha$ choose an equivariant 
fiber homotopy equivalence $h_\alpha\co S(\xi_\alpha)\to
S(f_\alpha^*\,\omega_\alpha)$, and let $k_\alpha$ be the composite
of $h_\alpha$ with the canonical induced bundle mapping
$S(f_\alpha^*\,\omega_\alpha)\to S(\omega_\alpha)$.  Define 
a map $H_\alpha\co 
D(M_\alpha)\to D(N_\alpha)$ using $k_\alpha$ and fiberwise radial
extension.  It follows that $H_\alpha$ is equivariantly homotopic
to $f|D(M_\alpha)$ for each $\alpha$, and hence by the equivariant
Homotopy Extension Property we may deform $f$ equivariantly to a map
$f_0$ such that $f_0|D(M_\alpha) = H_\alpha$ for each $\alpha$.
Since each $H_\alpha$ is isovariant, it follows that the restriction
of $f_0$ to a neighborhood of the fixed point set is isovariant as
required.
\end{proof}

We shall conclude this section by proving the sufficiency of the
condition in (iii).  A key step in the proof of \fullref{thm1} 
will be to prove that an equivariant homotopy equivalence has this
property.

\begin{prop}\label{prop42}
Let $M$ and $N$ be as before, and suppose that $f\co M\to N$ is a
continuous equivariant map that is isovariant on a neighborhood
of the fixed point set.  If there is a system of closed tubular 
neighborhoods $W_\alpha$ of $M^G$ such that the set of nonisovariant
points lies in the interiors of the sets $W_\alpha$, then $f$ is
equivariantly homotopic to an isovariant map.
\end{prop}

Note that we make no assumption about the images of the sets $W_\alpha$,
and in particular we do not assume that they lie in the tubular
neighborhoods of the components of the fixed point set of $N$.

\begin{proof}
By assumption $f$ is already isovariant on the 
closed complement of a submanifold $T$ of the form 
$\cup_\alpha~S(M_\alpha)\times [{\frac{1}{2}},1]$ where as usual
$S(M_\alpha)$ denotes the boundary of a tubular neighborhood.
Let $T_\alpha$ be the portion lying over $M_\alpha$ and denote
the boundary components corresponding to $S(M_\alpha)\times 
\{{\frac{1}{2}}\}$
and $S(M_\alpha\times \{1\}$ by $\partial_0T_\alpha$ and 
$\partial_1 T_\alpha$ respectively.  Let
$$B_\alpha~~=~~M~-~\cup_{\beta\neq \alpha}\;M_\beta\qquad\qquad 
C_\alpha~~=~~N~-~\cup_{\beta\neq \alpha}\;N_\beta$$
let $S_\alpha$ be the  spherical fiber of $S(M_\alpha)$,
and let $A_\alpha$ correspond to the annulus 
$S_\alpha\times [{\frac{1}{2}},1]$.

For each value of $j$ the map $f$ determines a map of triads
$$f_\alpha\co(T_\alpha;\partial_0T_\alpha,\partial_1
T_\alpha)~\longrightarrow~(C_\alpha;D(N_\alpha),N-N^G)$$
and by the results of Dula--Schultz \cite[Section~4]{DuS} it suffices 
to show that each such map of
triads can be compressed equivariantly rel $\partial_1
T_\alpha\cup A_\alpha$ into the triad
$(N-N^G;S(N_\alpha),N-N^G)$.  The 
methods of \cite[Sections~1 and~5]{DuS} imply that the 
obstructions to compression lie in diagram-theoretic Bredon 
equivariant cohomology groups of the form $\brh^i({\text{\bf T}};
\varPi_i)$, 
where {\bf T} is the diagram associated to the triad
$(T_\alpha;\partial_0T_\alpha,\partial_1T_\alpha\cup A_\alpha)$
and $\varPi_i$ is the following diagram of abelian groups:
$$
\begin{CD}
\pi_i(D(N_\alpha),S(N_\alpha))@>\partial_0^{\#}>>
\pi_i(C_\alpha,N-N^G)@<\partial_1^{\#}<<\{0\}
\end{CD}
$$
If $q_\alpha=\dim M-\dim F_\alpha$ then $(D(N_\alpha),
S(N_\alpha))$ 
is $(q_\alpha-1)$--connected
by the identity $\pi_s(D(N_\alpha),S(N_\alpha))
\cong\pi_{s-1}(S_\alpha)$, and
a standard general position argument shows that 
$(C_\alpha,N-N^G)=(C_\alpha,C_\alpha-N_\alpha)$
is also $(q_\alpha-1)$--connected.  Therefore the Blakers--Massey Theorem 
(see Gray \cite[page~143]{Gr}) implies
that the map $\partial^{\#}_0$ is bijective if 
$i\leq 2q_\alpha-3$ and surjective
if $i=2q_\alpha-2$.  In particular, this means that 
the equivariant diagram cohomology groups 
$\brh^i({\text{\bf T}};\varPi_i)$ reduce to ordinary 
Bredon cohomology groups $\brh^i(T_\alpha, 
\partial_1 T_\alpha;\pi_i(C_\alpha,N-N^G))$ if $i\leq n-1$
or if $i=n$ and $\dim F_\alpha+3\leq q_\alpha$.  
Since $T_\alpha\cong\partial_1 T_\alpha\times
[0,1]$ it follows immediately that the
relative cohomology groups vanish in all such cases.  Since 
$\dim T_\alpha= \dim M=\dim N$, this implies that the isovariance 
obstructions vanish in all cases except perhaps when 
$i=n=2q_\alpha-2$.  In such cases the value group fits
into the following exact sequence, which arises by restricting
diagram-theoretic cochains in $C(X'\to X; \pi'\to \pi)$ to ordinary 
cochains in $C(X';\pi')$:
$$
\begin{CD}
(\star)\qquad H^{n-1}\left(\partial_0 T_\alpha;\pi_n(D(\alpha_\alpha'),
S(\alpha_\alpha'))\,\right) @>\Delta >> H^{n}(T_\alpha,\partial T_\alpha;\pi_n(N_\alpha,N-N^G))\\
@. @VVV\\
@. \brh^n({\text{\bf T}};\varPi_n)\\
@. @VVV\\
@. H^{n}\left(\partial_0 T_\alpha;\pi_n(D(\alpha_\alpha'),S(\alpha_\alpha'))\,\right)
\end{CD}
$$
The map $\Delta$ in this sequence is given by combining the coefficient 
homomorphism for the map $\delta^{\#}_0$ in dimension $n$ with the
suspension isomorphism 
$$H^{n-1}(\partial_0 T_\alpha;\pi)~\longrightarrow~ H^{n}(T_\alpha,\partial 
T_\alpha;\pi)~.$$
Therefore the Blakers--Massey Theorem,
the $(n-1)$--dimensionality of $\partial_0 T_\alpha$, and the
Bockstein exact sequence for the short exact sequence
$$
\begin{CD}
0 @>>> {\text{\rm Kernel}} @>>>
\pi_n(D(\alpha_\alpha'),S(\alpha_\alpha'))@>\partial_0^{\#}>>
\pi_n(C_\alpha,N-N^G)@>>>0
\end{CD}
$$
imply that $\Delta$ is onto.  But the last object in $(\star)$ is zero 
because $\dim\partial_0T_\alpha=n-1$, and it follows by exactness that 
$\brh^n({\text{\bf T}};\varPi_n)$ is also trivial.  Therefore the 
isovariance obstructions always vanish.
\end{proof}

The following examples due to Browder \cite{Br2} show that it is not
always possible to deform an equivariant degree 1 map so that it
is isovariant near the fixed point set and the set of nonisovariant
points lies in a tubular neighborhood of the fixed point set.

\begin{exs}
Let $k$ and $q$ be distinct positive 
integers such that $q$ is even and $G$ has a free $q$--dimensional 
linear representation (that is, the group acts freely except
at the zero vector).  Let $N=S^k\times S^q$ with trivial action 
on the first coordinate and the one point compactification of the 
free linear representation on the second, let $M_0$ be the disjoint 
union of $N$ and two copies of the space
$G\times S^k\times S^q$ (where $G$ acts 
by translation on itself and trivially on the other two coordinates), 
and define an equivariant map $f_0\co M_0\to N$ by taking the identity 
on $M$, the unique equivariant extension of the identity map on  
$S^k\times S^q$ over one copy of $G\times S^k\times S^q$, and the unique
equivariant extension of an orientation reversing self diffeomorphism of
$S^k\times S^q$ over the other copy.  By construction this map has degree
one, and one can attach 1--handles equivariantly to $M_0$ away from the
fixed point set to obtain an equivariant cobordism of maps from $f_0$
to a map $f$ on a connected 1--manifold $M$ that is nonequivariantly 
diffeomorphic to a connected sum of $2\cdot|G|+1$ copies of $S^k\times
S^q$.  Since the fixed point sets of $M$ and $N$ are $k$--dimensional and the
manifolds themselves are $(k+q)$--dimensional, it follows that the Gap 
Hypothesis holds if we impose the stronger restriction $q\geq k+2$.  
By construction the map $f$ determines a homotopy equivalence of fixed 
point sets and is isovariant on a neighborhood of the fixed point set.
\end{exs}

\begin{claim}
It is not possible to deform $f$ equivariantly so
that the set of nonisovariant points lies in a tubular neighborhood of
the fixed point set.  In particular, it is also not possible to deform
$f$ equivariantly to an isovariant map.
\end{claim}

\begin{proof}[]
To prove the assertion, assume that one has a map $h$ equivariantly 
homotopic to $f$ with the stated property, and let $U$ be a tubular 
neighborhood of $M^G$ that contains the set of nonisovariant points.  
Let $X$ be a submanifold of the form $\{g\}\times\{v\}\times S^q$ in 
$M$ that arises from one of the copies of $G\times S^k\times S^q$ in 
$M_0$.  Although $X$ and $U$ may have points in common, by the 
uniqueness of tubular neighborhoods we can always isotope $X$ 
into a submanifold $X'$ that is disjoint from $U$.  By the hypotheses 
on $h$ we know that $h(X')$ is disjoint from $N^G=S^k\times S^0$, and 
therefore $h(X')$ is contained in 
$$N - N^G~~\cong~~S^k\times S^{q-1}\times {\mathbb R}$$
so that the image of the generator of $H_q(X') = {\mathbb Z}$ maps
trivially into $H_q(N)$.  However, $h$ is supposed to be homotopic
to a map which is nontrivial on the latter by construction, so we
have a contradiction, and therefore it is not possible to find
an isovariant map $h$ that is equivariantly homotopic to 
$f$.
\end{proof}

A refinement of the preceding argument shows that if $Y$ is a
subset of $M$ such that the image of $H_q(Y)$ in $H_q(M)$ is
equal to the image of $H_q(X)$, then $Y$ must contain some
nonisovariant points of any equivariant map that is equivariantly
homotopic to $f$.

\begin{rem}
By construction, Browder's examples are normally cobordant to the
identity; an explicit normal cobordism from the identity to $f_0$
is given by 
$$W~~=~~N\times [0,1]~\amalg~ G\times S^k\times S^q\times [0,1]$$
where $\partial_-W = N\times \{0\}$ and $\partial_+W$ is the
remaining $2|G|+1$ components of the boundary, and one can obtain
a normal cobordism to $f$ by adding 1--handles equivariantly along
the top part of the boundary.  More generally, results of K H 
Dovermann \cite{Dov} imply that one can always construct equivariant
normal cobordisms to equivariant normal maps if the Gap Hypothesis
holds and the map is already an equivariant homotopy equivalence 
on the singular set as in Browder's examples.
\end{rem}

However, it is also possible to construct examples like Browder's
that are not cobordant to highly connected maps.  It suffices
to let $k \equiv 0 (4)$ and replace $G\times S^k\times S^q$ by
$G\times S(\gamma)$, where the latter is the sphere bundle of
a fiber homotopically trivial vector bundle $\gamma$ over $S^k$
with nontrivial rational Pontryagin classes; one must also
replace the equivariant folding map from $G\times S^k\times S^q$
to $N$ by its composite with the identity on $G$ times a
fiber homotopy equivalence from $S(\gamma)$ to $S^k\times
S^q$.  Characteristic number arguments imply the map obtained
in this fashion is not cobordant to a $k$--connected map.  Of course, 
a degree 1 map of this type does not have the bundle data required 
for a normal map in the sense of equivariant surgery theory.

\subsection*{Generalizations to nonsemifree actions}

We have assumed our actions are semifree in order to keep the
discussion of normal straightening as simple as possible.  In
fact, the main obstruction to proving analogs of Theorems 
\ref{thm1}, \ref{thm2}, and \ref{thm3}
for arbitrary actions is to prove a suitable generalization
of the result on normal straightening.  One step in carrying
this out is to develop some methods for analyzing the 
equivariant fiber homotopy type of certain equivariant vector
bundles.  In the semifree case the base spaces for these
bundles all have trivial $G$--actions, and we were able to
study the bundles directly with a minimum of equivariant 
homotopy theoretic machinery.  However, if the action on the
ambient manifold is not semifree, then we must also consider
equivariant vector bundles over the various fixed point sets
of proper subgroups of $G$, and these bases will usually have
nontrivial group actions.  Working with such objects requires
the sorts of constructions developed by Schultz \cite{Sc4,Sc5} as well
as the full force of S Waner's work on
the equivariant fiber homotopy types of equivariant fiber 
bundles \cite{Wn}.
If $G = {\mathbb Z}_{p^r}$ where $p$ is prime, then the generalization
is particularly direct, and with such results one can prove
complete generalizations of Theorems 
\ref{thm1}, \ref{thm2}, and \ref{thm3} for actions of
such groups (assuming the Gap Hypothesis as usual).

\section{Equivariant homotopy equivalences}\label{equivar}

In this section we shall show that an equivariant homotopy
equivalence can be deformed to satisfy the conditions in parts
(ii) and (iii) of \fullref{thm2} and thus must be 
equivariantly homotopic to an isovariant map, which we shall 
prove must be an isovariant homotopy equivalence.  

\begin{prop}\label{prop51}
Let $f\co M\to N$ be a homotopy equivalence of closed, oriented, 
semifree $G$--manifolds which satisfy the
Gap Hypothesis such that all components of all fixed point 
sets are also orientable.  Then $f$ is equivariantly homotopic
to a map that is isovariant on a neighborhood of the fixed  
point set.
\end{prop}

\begin{proof}
We shall prove that $f$ satisfies the conditions in part (ii)
of \fullref{thm2}.  Since $f$ defines a homotopy equivalence of
fixed point sets, it follows immediately that for each component
$M_\alpha$ of $M_G$ the restriction $f_\alpha$ of $f$ defines
a homotopy equivalence from $M_\alpha$ to $N_\alpha$ and hence
has degree $\pm 1$.  In order to apply part (ii) of \fullref{thm2},
we also need to verify the homotopy pullback condition on the
equivariant normal bundles of the fixed point set components.

Let $\tau_M$ and $\tau_N$ be the equivariant tangent 
bundles of $M$ and $N$.  We claim that the sphere bundles of 
$\tau_M$ and $f^*\tau_N$ are stably equivariantly fiber homotopically 
equivalent.  The nonequivariant version of this statement is well 
known (cf Atiyah \cite{At}) and the equivariant case is 
due to K Kawakubo \cite{Ka}.

Consider next the restriction of the stable equivariant fiber homotopy 
equivalence $S(\tau_M)\sim S(f^*\tau_N)$ to $M^G$.  The classifying maps
for the two equivariant fibrations go from $M^G$ to a space ${\mathcal B}$ 
such that $\pi_*({\mathcal B})\approx\pi^G_{*-1}$, where the latter denotes an
equivariant stable homotopy group as in Segal \cite{Se}.  On the other hand, 
by \cite{Se} we also know that ${\mathcal B}$ is homotopy equivalent to the
product $BF\times BF_G$  where $BF$ classifies nonequivariant stable spherical
fibrations and $BF_G$ is defined as in Becker--Schultz \cite{BeS}.  In terms of
fibrations the projections of the classifying maps $M\to {\mathcal B}$ onto 
$BF$ and $BF_G$ correspond to taking the classifying maps of the fixed point 
subbundles and the orthogonal complements of the fixed point subbundles
respectively.  Therefore it follows that the corresponding subbundles for
$\tau_M$ and $f^*\tau_N$ are stably equivariantly fiber homotopy equivalent.
In particular, this means that $S(\alpha_M)$ and $S([f^G]^*\alpha_N)$ are 
stably equivariantly fiber homotopy equivalent because they induce 
homotopic maps from $M^G$ into $BF_G$.

As usual, write $M^G$ as a disjoint union of components 
$M_\alpha$, and for each $\alpha$
let $q_\alpha$ be the codimension of $M_\alpha$.  Furthermore, 
denote the fiber representation at a point of $M_\alpha $ by 
$V_\alpha$.  The stabilization map $F_G(V_\alpha)\to F_G$ is 
$(q_\alpha-2)$--connected by the considerations of
[Sc1], and the Gap Hypothesis implies that $\dim M_\alpha\leq q_\alpha-2$.
Therefore we can destabilize the stable fiber homotopy equivalence from
$S(\xi_\alpha)$ to $S([f^G]^*\omega_\alpha)$ and obtain a genuine 
equivariant fiber homotopy  equivalence.  Choose such an equivariant 
fiber homotopy equivalence, say $\Phi$.  It is then an elementary exercise 
to deform $f|\cup_\alpha D(M_\alpha)$ equivariantly relative the zero 
section so that one obtains the radial extension of $\Phi$ at the 
other end of the deformation.  By the equivariant homotopy extension 
property one can extend this homotopy to all of $M$.
\end{proof}

Our choice of fiber homotopy equivalences was arbitrary, but it is
possible to find a canonical choice up to homotopy using equivariant
$S$--duality (see Wirthm\"uller \cite{Wi}) and Kawakubo's result; in
fact, one must work with the latter to prove a uniqueness result for
isovariant deformations as in \cite{Br2}, and we shall explain this in
a subsequent article.

The preceding result and part (iii) of \fullref{thm2} reduce 
the proof of Theorems \ref{thm1} and \ref{thm3} to the following 
two results:

\begin{prop}\label{prop52}
Suppose that $f$ satisfies the conditions of the previous result,
including the condition that $f$ is isovariant on a neighborhood
of the fixed point set.  Then $f$ is homotopic to an almost
isovariant map.
\end{prop}

\begin{prop}\label{prop53}
If $f$ as above is isovariant, then $f$ is an isovariant homotopy
equivalence.
\end{prop}

We shall prove these results in order.

\begin{proof}[Proof of \fullref{prop52}]
The first step is to construct an equivariant homotopy from $f$
to a mapping $f_1$ such that the homotopy is fixed on a neighborhood
of the fixed point set and $f_1$ is smoothly equivariantly transverse 
to the fixed point set; there are no problems with equivariant 
transversality obstructions because the relevant part of the domain 
has a free $G$--action.  By construction the transverse inverse image 
of the fixed point set is the set of nonisovariant points, and it 
splits into a union of smooth submanifolds $V_\alpha = f_1^{-1}
(N_\alpha)$.  Note that $\dim V_\alpha = \dim N_\alpha = \dim
M_\alpha$ which is less than half the dimensions of $M$ and $N$.

By construction the image of $f_1|V_\alpha$ is contained in  
$N_\alpha$, so let $k_\alpha$ be the associated map from 
$V_\alpha$ to $N_\alpha$;  if $h_\alpha\co N_\alpha\to M_\alpha$ 
is determined by a homotopy inverse to $f_1$, then the map 
$h_\alpha\tinycirc k_\alpha$ is equivariantly homotopic to the 
inclusion of $V_\alpha$ in $M$. By general position it follows
that the latter is also equivariantly homotopic to a map into
$D(M_\alpha) - M_\alpha$ and in fact can be approximated by a
smooth equivariant embedding $e_\alpha$; in fact, the numerical 
condition in the Gap Hypothesis is strong enough to guarantee 
that $e_\alpha$ is equivariantly isotopic to the inclusion.  
Since the image of $e_\alpha$ is contained in a tubular neighborhood
of $M_\alpha$, the Equivariant Isotopy Extension Theorem implies
the inclusion is isotopic to a smooth equivariant embedding of 
$V_\alpha$ in a tubular neighborhood and hence the image of the
inclusion itself must also be contained in some tubular 
neighborhood.  Since this is true for every $\alpha$, it is
also true for the entire set of nonisovariant points.
\end{proof}

By \fullref{thm2} and the preceding propositions we know 
that $f$ is equivariantly homotopic to an almost isovariant 
mapping, and by \cite{DuS} it is also equivariantly homotopic 
to an isovariant mapping.

\begin{proof}[Proof of \fullref{prop53}]
By \cite[Proposition 4.1, page 27]{DuS}, the map $f$ is 
isovariantly homotopic to a map $f_0$ such that for each 
$\alpha$ we have $f(\X D(M_\alpha)\,)\subset D(N_\alpha)$,
$f(\X S(M_\alpha\,))\subset S(N_\alpha)$, and
$$f\left(\X M - \cup_\alpha~{\text{\rm Int\,}} D(M_\alpha)\X
\right)~~\subset~~ N - \cup_\alpha~{\text{\rm Int\,}} D(N_\alpha)~.$$
Furthermore, using \cite[Theorem~4.4, pages~29--31]{DuS} one can
further deform this map to some $f_1$ that is fiber preserving on the 
tubular neighborhoods and maps disk fibers to disk fibers by cones 
of maps over the boundary spheres (that is, the map is
\emph{normally straightened} in the sense of \cite[page 31]{DuS}).
It will suffice to prove that $f_1$ is an isovariant homotopy
equivalence, so without loss of generality we might as well 
assume that $f$ itself is normally straightened.

By the isovariant Whitehead Theorem established in Section 4
of \cite{DuS}, the map $f$ is an isovariant homotopy
equivalence if $f$ defines a homotopy equivalence from $M-M^G$ to
$N-N^G$.  General position considerations imply that $f$ induces an
isomorphism of fundamental groups, and therefore it suffices to check 
that $f$ defines an isomorphism in homology with twisted coefficients 
in the group ring of the fundamental group.  Exact sequence and excision
arguments show that
the latter holds if $f$ induces homotopy equivalences from $M$ to $N$,
from $M^G$ to $N^G$, and from $\coprod S(\xi_\alpha)$ to 
$\coprod S(\omega_\alpha)$.  The first two of these follow 
because $f$ is an 
equivariant homotopy equivalence.  To prove the third property first 
note that for each $\alpha$ the homotopy fibers of $S(\xi_\alpha)\subset 
D(\xi_\alpha)$ and $S(\omega_\alpha)\subset D(\omega_\alpha)$ are simply 
the fibers of the sphere bundles; since each $D(\xi_\alpha)$ maps to 
$D(\omega_\alpha)$ by a homotopy equivalence, it suffices to know that 
a fiber of $S(\xi_\alpha)$ maps to a fiber of $S(\omega_\alpha)$ with 
degree $\pm 1$.  This follows directly from the construction of the 
isovariant map; the first step was to make an equivariant homotopy 
equivalence normally straightened near the fixed point set, and the 
equivariant deformation in part (iii) of \fullref{thm2} is constant
near some fiber of $S(\xi_\alpha)$.
\end{proof}

\subsection*{Recognizing isovariant homotopy equivalences}

One can combine \fullref{prop53} with \fullref{thm1} 
to obtain the following specialization of the Isovariant Whitehead 
Theorem: 

\begin{thm}\label{thm54}
Let $f\co M\to N$ be a continuous isovariant
mapping of oriented closed semifree smooth $G$--manifolds that 
satisfy the Gap Hypothesis.  Then $f$ is an isovariant homotopy
equivalence if and only if $f$ is an ordinary homotopy equivalence
(forgetting the group action), and the associated map of fixed 
point sets $f^G\co M^G\to N^G$ is also a homotopy equivalence.  
\end{thm}

\begin{proof}[Proof of \fullref{thm54}]
The necessity of the conditions is immediate, so we only need
to check that the latter are also sufficient.
Since smooth $G$--manifolds have the equivariant homotopy types
of $G$--CW complexes (cf Illman \cite{IL2}), the Equivariant 
Whitehead Theorem of Illman \cite{IL1} and T Matumoto \cite{Ma} 
and the hypotheses imply that $f$ must be an equivariant homotopy 
equivalence.  Since $f$ is an isovariant map, it follows from 
\fullref{prop53} that $f$ is an
isovariant homotopy equivalence.
\end{proof}

\begin{rem}[Extension to nonsemifree actions]
We have mentioned that \fullref{thm1} extends to at least some smooth
actions that are not semifree, and in fact the same is true for
\fullref{prop53}.  Given these, one can also extend \fullref{thm54}
provided the final clause about $f^G$ is replaced with the 
following:  \emph{\ldots and for each subgroup $H\subset G$ the
associated map of $H$--fixed point sets $f^H\co M^H\to N^H$ is also
an ordinary homotopy equivalence.\/} (Note: since the group $G$ 
acts trivially on its fixed point set, there is no difference 
between equivariant and ordinary homotopy equivalences of the
singular sets --- that is, the points where the action is
not free --- in the semifree case.)
\end{rem}

\section{Final remarks}\label{FinRem}

Questions about the role of the Gap Hypothesis in transformation
groups have been around for some time (cf Schultz \cite{ScE}).  Such
questions rarely have clear cut answers, but the main results
of this paper strongly suggest that the usefulness of the Gap 
Hypothesis is closely related to 

\begin{enumerate}
\item the strong implications of isovariance for 
analyzing existence and classifications questions for group actions,

\item some very close relationships between isovariant
homotopy and equivariant homotopy when the Gap Hypothesis holds.
\end{enumerate}

\noindent
We shall discuss a few aspects of these points in this section.

\subsection*[Extending \ref{thm1} to other cases]
{Extending \fullref{thm1} to other cases}

Since the Gap Hypothesis was used at several crucial points in
the proof of our main theorems, one might reasonably expect that
these results do not necessarily hold if the Gap Hypothesis fails.
Despite this, there are some situations in which one can prove
analogs of \fullref{thm1}, particularly when $G$ is cyclic 
of prime order and  the difference $\dim M - 2\dim M^G$ is equal 
to 1 or 0.  In particular,
if we also assume that $G$ is cyclic of prime order and there is
only one component with maximal dimension, then one can use surgery
theory to prove generalizations of \fullref{thm1}.  If $G$ has order 2
and the dimension difference is zero, then this is established in 
\cite{Sc2}, and the other cases will be shown in a forthcoming paper by
K H Dovermann and the author.  On the other hand, the previously
quoted results of results of Longoni and Salvatore \cite{LS} imply that 
the analog of \fullref{thm1} does not necessarily hold if $G$ has 
order 2 and $M$ is not simply connected.

In a subsequent paper we shall use equivariant function spaces as 
in \cite{Sc1} and \cite{BeS} to construct systematic families of 
equivariant homotopy equivalences that are not homotopic to
isovariant maps in situations where the Gap Hypothesis fails.
In particular, we shall construct connected examples where $G$ 
is cyclic of prime order, $\dim M = 2\dim M^G$, and there are 
two components with the maximal dimension.

\subsection*{Implications for equivariant surgery}

The methods and results of \cite{DuS} provide a means for 
analyzing isovariant homotopy theory --- and its relation to 
equivariant homotopy theory --- within the standard framework 
of algebraic topology.  Therefore \fullref{thm1} and the 
conclusions of \cite{DuS} suggest a two step 
approach to analyzing smooth $G$--manifolds within a given equivariant 
homotopy type if the Gap Hypothesis does not necessarily hold; namely, 
the first step is to study the obstructions to isovariance for an 
equivariant homotopy equivalence and the second step is to study one 
of the versions of the isovariant surgery theories in \cite{Sc3} or 
\cite{We}.  This approach seems especially promising for analyzing 
classification questions by means of surgery theory and homotopy theory.

There are also other indications of very close ties between the validity
of \fullref{thm1} and potential extensions of equivariant surgery to
cases outside the Gap Hypothesis range.  The extensions of \fullref{thm1}
mentioned in the preceding paragraphs use results on equivariant surgery
just outside the Gap Hypothesis range, and in particular they depend
heavily on some corresponding extensions of results from equivariant
surgery.  Of course, one would also like to have more homotopy theoretic
proofs for such results.  On the other hand, the counterexamples to
\fullref{thm1} in certain cases also imply that one cannot expect to
have purely algebraic obstructions to equivariant surgery (based upon
data like the equivariant fundamental group) in some situations that
are very close to those considered by Dovermann \cite{D0} or by A Bak
and M Morimoto \cite{BM1,BM2}.  This seems to reflect the failure of the
equivariant $\pi-\pi$ theorem (compare Dovermann--Rothenberg \cite{DR})
in certain cases where the Gap Hypothesis is not valid (see the last
section of Chapter I in Dovermann--Schultz \cite{DoS} for quite different
examples along these lines).

\subsection*{A dual approach}

The results of this paper also raise questions about \emph{a priori}
descriptions for isovariant homotopy types of semifree action in 
equivariant terms when the gap hypothesis holds.  \fullref{thm54} is
a simple observation in this direction.  Another way of looking
at the question is to consider the implications of the main theorem
for a dual approach to the problem in which one begins by trying to
work with the free parts of the group actions on the manifolds;
suh an approach might be particularly useful if one wishes to 
generalize the main result to nonsmoothable actions using tools
from controlled topology.  If we are given two finite equivariant CW
complexes $M$ and $N$ with semifree actions of the group $G$ that
are equivariantly homotopy equivalent and their orbit spaces are 
given by $M^*$ and $N^*$ respectively, then purely formal considerations 
imply that the one point compactifications of $M^* - M^G$ and $N^*- N^G$
are homotopy equivalent because these spaces are homeomorphic to the
quotient spaces $M^*/M^G$ and $N^*/N^G$ respectively.  If $M$ and $N$ 
are smooth manifolds and the Gap Hypothesis holds, then \fullref{thm1} implies that this canonical homotopy equivalence of one point
compactifications comes from a proper homotopy equivalence between
$M^* - M^G$ and $N^*- N^G$ and this mapping can be compactified
in the sense of Quinn's work on ends of maps \cite{Q}.  A good 
understanding of this from the viewpoint of controlled topology 
as in \cite{Q} would be extremely useful for studying possible
extensions of \fullref{thm1} to group actions that are not 
necessarily smooth but are still somehow well behaved topologically.

\newpage
\bibliographystyle{gtart}
\bibliography{link}

\end{document}